\documentclass[
a4paper,							
12pt,								  
]
{scrartcl}

\usepackage[ngerman, english]{babel} 
\usepackage[T1]{fontenc}
\usepackage[ansinew]{inputenc}

\usepackage{lmodern} 

\usepackage[hyphens,obeyspaces,spaces]{url}

\usepackage{color}

\usepackage{amssymb}
\usepackage{mathrsfs} 

\usepackage{amsmath}
\usepackage{amsthm}
\usepackage{amsfonts}

\usepackage{paralist}
\usepackage{stmaryrd}
\usepackage{caption}

\usepackage[arrow, matrix, curve]{xy}

\usepackage{graphicx} 

\newtheorem{satz}{Theorem}[section]
\newtheorem{lemma}[satz]{Lemma}
\newtheorem{prop}[satz]{Proposition}
\newtheorem{korollar}[satz]{Corollary}
\newtheorem{bem}[satz]{Remark}
\newtheorem{definition}[satz]{Definition}

\theoremstyle{definition}

\newcommand{\A}{\mathbb{A}}
\newcommand{\N}{\mathbb{N}}

\newcommand{\Z}{\mathbb{Z}}

\newcommand{\vp}{\varphi}

\DeclareMathOperator{\id}{id}

\DeclareMathOperator{\Hom}{Hom}

\DeclareMathOperator{\Aut}{Aut}

\DeclareMathOperator{\Mod}{Mod}

\begin{document}
\title{JSJ decompositions of doubles of free groups}
\author{Simon Heil\\
Math. Seminar, Christian-Albrechts-University \\
Kiel, Germany\\
email: heil@math.uni-kiel.de}
\maketitle

\begin{abstract}
\noindent We classify all possible JSJ decompositions of doubles of free groups of rank two and we also compute the Makanin-Razborov diagram of a particular double of a free group and deduce that in general limit groups are not freely subgroup separable.
\end{abstract}

\section{Introduction}
A group is called \textit{subgroup separable} if for every finitely generated subgroup $H\leq G$ and $g\in G\setminus H$ there exists a homomorphism $\vp:G\to E$ to a finite group $E$ such that $\vp(H)=1$ and $\vp(g)\neq 1$. Subgroup separability was shown for polycyclic groups by Mal'cev \cite{mal'cev}, for free groups by M. Hall \cite{hall}, for surface  groups by P. Scott \cite{scott} and more recently for limit groups by H. Wilton \cite{wilton}.\\
Now one could wonder if it is possible to achieve Wilton's result by passing to a free quotient of the limit group and then use the result of M. Hall. This was indeed phrased as a question on mathoverflow.com \cite{mathoverflow} by K. Bou-Rabee, namely whether or not a limit group  is freely subgroup separable. A group $L$ is called \textit{freely subgroup separable} if for every finitely generated subgroup $H\leq L$ and every element $g\in L\setminus H$ there exists a homomorphism $\vp$ from $L$ to a free group, such that $\vp(g)\notin \vp(H)$.\\
If limit groups would have this property, then as mentioned before, together with the result of M. Hall, one could deduce the theorem of Henry Wilton that limit groups are subgroup separable.
First I. Agol notes in the discussion on matheoverflow.com that in order to have a chance of achieving a positive result, one has to assume that the subgroup $H$ is of infinite index in $L$.\\ 
Henry Wilton then conjectured that one can show that limit groups are not freely subgroup separable by computing the Makanin-Razborov diagram of a particular double of a free group (which is a limit group). Roughly spoken a Makanin-Razborov diagram of a finitely generated group $G$ is a finite directed tree that encodes all homomorphisms from $G$ to a non-abelian free group $F$ by yielding a parametrization of $\Hom(G,F)$. We give a precise definition later on.\\
We are following this approach and try to describe Makanin-Razborov diagrams of doubles of a free group in two generators. 
First classes of examples of Makanin-Razborov diagrams of such doubles, respectively the limit groups appearing in these diagrams were computed by Nicholas Touikan in \cite{touikan2}, \cite{touikan3}. We refer the reader to his papers for a broader background on the subject.\\
We compute the Makanin-Razborov diagram of a particular double of a free group of rank two, not covered by the work of Touikan and use this diagram to show that limit groups are in general not freely subgroup separable (even if one assumes that the subgroup $H$ is of infinite index, as mentioned by Agol). After the first version of our paper this double and its MR diagram were independently constructed and used by Louder and Touikan to show that limit groups are also not freely conjugacy separable (\cite{loudertouikan}).\\

The first step in constructing a MR diagram of a given f.g. (one-ended) group $L$, is to understand all splittings of $L$ along cyclic subgroups, that is to find a splitting of $L$ as a graph of groups which encodes in some sense all of these splittings. Such a graph of groups is called a JSJ decomposition of $L$. To make this more precise let $\mathcal{Z}$ be the set of all infinite cyclic subgroups of $L$. Given two $\mathcal{Z}$-trees $T_1$ and $T_2$, i.e. simplicial trees on which $L$ acts with edge stabilizers in $\mathcal{Z}$, we say that $T_1$ \textit{dominates} $T_2$ if any group which is elliptic in $T_1$ is also elliptic in $T_2$. A $\mathcal{Z}$-tree is \textit{universally elliptic} if its edge stabilizers are elliptic in every $\mathcal{Z}$-tree.

\begin{definition}
Let $L$ be a finitely generated group and $T$ a $\mathcal{Z}$-tree such that
\begin{enumerate}[(a)]
\item $T$ is universally elliptic and
\item $T$ dominates any other universally elliptic $\mathcal{Z}$-tree $T'$. 
\end{enumerate}
We call $T$ a \textnormal{JSJ tree} and the quotient graph of groups $\A=T/L$ a \textnormal{cyclic JSJ decomposition} (or for short \textnormal{JSJ decomposition}) of $L$.
\end{definition}

For arbitrary finitely generated groups JSJ decompositions do not always exists, but since in our case $L$ is a one-ended limit group and therefore in particular finitely presented, a JSJ decomposition of $L$ does exist. The existence and much more about JSJ decompositions (not only along cyclic edge groups) can be found in \cite{guirardel}.\\
Unfortunately JSJ decompositions are not unique in general, but rather form a deformation space, denoted by $\mathcal{D}_{JSJ}$, which consists of all cyclic JSJ trees of $L$. Hence two JSJ trees $T_1,T_2$ are in $\mathcal{D}_{JSJ}$ if and only if they have the same elliptic subgroups.\\
Still in some cases there exists a canonical JSJ decomposition, i.e. a decomposition that is invariant under automorphisms. In particular when $L$ is a one-ended limit group one can apply the \textit{tree of cylinder construction} of Guirardel and Levitt (see section 7 and Theorem 9.5 in \cite{guirardel}) to get a canonical cyclic JSJ decomposition.\\

Let now $\A$ be a cyclic JSJ decomposition of a finitely generated group $L$. A vertex group $A_v$ of $\A$ is called \textit{rigid} if it is elliptic in every splitting of $L$ along a cyclic subgroup and \textit{flexible} otherwise. The description of flexible vertices of JSJ decompositions (along a given class of groups) is one of the major difficulties in JSJ theory. In the case that $L$ is a one-ended limit group the structure of flexible vertex groups of a cyclic JSJ decomposition $\A$ is well-understood, namely these are either free abelian or \textit{quadratically hanging (QH) vertex groups}. A vertex group $A_v$ of $\A$ is QH if $A_v$ is isomorphic to the fundamental group of a compact surface $\Sigma$ with boundary in such a way that any incident edge group can be conjugated into a boundary subgroup of $\pi_1(\Sigma)$.\\

Now we are ready to state the formal definition of a Makanin-Razborov diagram.

\begin{definition}\label{MRdiagram}
Let $F$ be a non-abelian free group and $G$ be a finitely generated group. A finite directed rooted tree $T$ with root $v_0$ is called a \textnormal{Makanin-Razborov diagram} for $G$ if it satisfies the following conditions:
\begin{enumerate}[(a)]
\item The vertex $v_0$ is labeled by $G$.
\item Any vertex $v\in VT$, $v\neq v_0$, is labeled by a limit group $G_v$.
\item Any edge $e\in ET$ is labeled by an epimorphism $\pi_e:G_{\alpha(e)}\to G_{\omega(e)}$ such that for any homomorphism $\vp:G\to F$ there exists a directed path $e_1,\ldots,e_k$ from $v_0$ to some vertex $\omega(e_k)$ such that 
$$\vp=\Psi\circ\pi_{e_k}\circ\alpha_{k-1}\circ\pi_{e_{k-1}}\circ\ldots\circ\alpha_1\circ\pi_{e_1}$$
where $\alpha_i\in \Mod G_{\omega(e_i)}$ for $1\leq i\leq k$ and $\Psi$ is injective.
\end{enumerate}
\end{definition}

The modular group $\Mod(L)$ of a one-ended limit group $L$ is the subgroup of $\Aut(L)$ generated by Dehn twists along edges and extensions of automorphisms of flexible vertex groups of the (canonical) cyclic JSJ decomposition of $L$.
For more information on Makanin-Razborov diagrams and a formal definition of the modular group of a limit group see chapter 5 of \cite{sela}.\\

In the second section we introduce particular words in free groups, so-called $C$-test words, and compute the Makanin-Razborov diagram of a double of a free group of rank two along such a $C$-test word. We then use this diagram to show that limit groups are in general not freely subgroup separable. Unfortunately we were not able to give a description of all possible Makanin-Razborov diagrams of general doubles of free groups but at least we are able to classify all possible JSJ decompositions of these doubles in the third section. This is the first step towards the construction of more examples of MR diagrams or even to the possible classification of all MR diagrams of doubles of free groups of rank two in the future.\\
To compute the JSJ decompositions of a double of $F_2$ along a given word $w$ it is necessary to understand relative cyclic JSJ decompositions of the non-abelian free group $F_2$ relative to $w$. In \cite{cashen} C. Cashen has shown that for an arbitrary non-abelian free group $F$ there exists a canonical cyclic JSJ decomposition relative to finitely many maximal cyclic subgroups of $F$. Moreover given finitely many such subgroups Cashen and Manning have written an algorithm that computes this relative JSJ decomposition (\cite{cashenprogram}).\\

The results of this article are part of my doctoral thesis, which I am currently writing at the University of Kiel. I would like to thank my advisor Richard Weidmann for fruitful discussions, his
continuous support and helpful comments on an earlier version of this paper.

\section{Makanin-Razborov diagrams of doubles of free groups}
We first fix some notation.
Let $F_2=\langle x_1,x_2\rangle$ be a non-abelian free group, $w\in F_2$ and $G_w$ the double of $F_2$ along $\Z$ given by the embedding $\iota:\Z\to F_2, 1\mapsto w$.
In the following we denote such a double decomposition of $G_w$ along $w$ by
$G_w:=A\ast_C B$ where $A=F_2(a_1,a_2)$, $B=F_2(b_1,b_2)$, $C=\langle c\rangle\cong\Z$ with the embeddings $\iota_A:  C\to A: c\mapsto w_A$, 
$\iota_B:  C\to A: c\mapsto w_B$, where $w_A$ ($w_B$) is the image of $w$ under the canonical isomorphisms from $F(x_1,x_2)$ to $A$ ($B$) given by $x_i\mapsto a_i$ ($x_i\mapsto b_i$), $i\in\{1,2\}$.\\
We are now going to describe the Makanin-Razborov diagram of $G_w$ for a very specific word, a so-called C-test word:

\begin{definition}\cite{ivanov} 
Let $F_n$ be a non-abelian free group in $n$ generators.
A non-trivial word $w\in F_n$ is a C-test word in $n$ letters if for any two $n$-tuples $(A_1,\ldots,A_n)$, $(B_1,\ldots,B_n)$ of elements of a non-abelian free group $F$ the equality 
$$w(A_1,\ldots,A_n)=w(B_1,\ldots,B_n)\neq 1$$
implies the existence of an element $S\in F$ such that $SA_iS^{-1}=B_i$ for all $i\in\{1,\ldots,n\}$.
\end{definition}

\begin{bem} \label{bem}
Let $F_n$ be a free group of rank $n\geq 2$.
\begin{enumerate}
\item The definition of a C-test word does not depend on the rank of the free group $F$. This follows from Sela's solution of the Tarski problems, since being a C-test word $w$ in $n$ letters can be expressed by the following first-order sentence:
\begin{align*}
\forall a_1,\ldots,a_n\forall b_1,\ldots,b_n\exists s:\ &(w(a_1,\ldots,a_n)=w(b_1,\ldots,b_n)\neq 1)\\
&\Rightarrow(a_1=sb_1s^{-1}\wedge\ldots\wedge a_n=sb_ns^{-1})
\end{align*}
\item \cite{turner} A C-test word $w\in F_n$ is not contained in a proper retract of $F_n$.
\item \cite{ivanov} Every C-test word $w\in F_n$ is contained in the commutator subgroup of $F_n$.
\item \cite{ivanov} If a C-test word $w\in F_n$ is not a proper power, then the stabilizer of $w$ in $\Aut(F_n)$ is $\langle c_w\rangle$, where $c_w$ denotes conjugation by $w$.
\end{enumerate}
\end{bem}

\noindent The following result is due to Ivanov.

\begin{satz}\cite{ivanov}
For arbitrary $n\geq 2$ there exists a C-test word $w_n\in F_n$. In addition $w_n$ is not a proper power.
\end{satz}

D. Lee in \cite{lee} generalized this result to the following:

\begin{satz}\label{testelement}\cite{lee}
For arbitrary $n\geq 2$ there exists a word $w_n\in F_n$ that is a C-test word in $n$ letters such that $w_n$ is not a proper power and with the additional property that for elements $A_1,\ldots,A_n$ in a free group $F$ the following is equivalent:
\begin{enumerate}[(a)]
\item $w_n(A_1,\ldots,A_n)=1$.
\item The subgroup $\langle A_1,\ldots,A_n\rangle$ of $F$ is cyclic.
\end{enumerate}
\end{satz}

Hence we get the following corollary.

\begin{korollar}\label{test}\cite{lee}
There exists an element $w\in F_n$ such that if $\vp$ is an endomorphism of $F_n$, $\Psi$ is an endomorphism of $F_n$ with non-cyclic image, and $\vp(w) = \Psi(w)$, then $\vp = c_S \circ\Psi$ for some $S \in F_n$ such that $\langle S,\Psi(w)\rangle\cong \Z$.
If $n=2$, then $S\in\langle \Psi(w)\rangle$.
\end{korollar}

\begin{definition}
 We call a C-test word which satisfies the assumptions of Theorem \ref{testelement} and Corollary \ref{test} an \textnormal{Ivanov word}.
\end{definition}

The following corollary follows immediately from the proof of the above theorem in \cite{lee}.

\begin{korollar}\label{ivanov}
Let $F_2=\langle x_1,x_2\rangle$. Then
\begin{align*}
w=&[x^8_1,x^8_2]^{100}x_1[x^8_1,x^8_2]^{200}x_1[x^8_1,x^8_2]^{300}x_1^{-1}[x^8_1,x^8_2]^{400}x_1^{-1}\\
&\cdot[x^8_1,x^8_2]^{500}x_2[x^8_1,x^8_2]^{600}x_2[x^8_1,x^8_2]^{700}x_2^{-1}[x^8_1,x^8_2]^{800}x_2^{-1}
\end{align*}
is an Ivanov word.
\end{korollar}

\noindent Before we start the construction of a Makanin-Razborov diagram of a double $G_w$ along an Ivanov word $w$ we make the following observation.

\begin{lemma}\label{one-ended}
Let $F_n$ be a non-abelian free group, $w\in F_n$ and $G_w$ the double of $F_n$ along $w$. $G_w$ is one-ended if and only if $w$ is not contained in a free factor of $F_n$.
\end{lemma}

\begin{proof}
By Corollary 1.5 in \cite{touikan} $G_w$ is one-ended if and only if $A$ is one-ended relative to $\iota_A(C)$, i.e. if and only if there does not exist a free splitting $A_1\ast A_2$ of $A$ such that $\iota_A(C)$ is contained in $A_1$.
\end{proof}

The following theorem amounts to a computation of the Makanin-Razborov diagram of a double of $F_2$ along an Ivanov word.

\begin{satz}\label{MR}
Let $w \in F_2$ be an Ivanov word and $G_w=A\ast_{\langle w\rangle}B$ be the double of $F_2$ along $w$. Then after precomposition with a Dehn twist, every homomorphism from $G_w$ to a non-abelian free group factors through either the canonical retraction $\pi:G_w\to A$ or the projection $\eta:G_w\to \Z^2\ast\Z^2$.
\end{satz}

\begin{proof}
First note that by Remark \ref{bem} $w$ is not contained in a proper retract of $F_2$. In particular $w$ is not contained in a free factor of $F_2$ and hence $G_w$ is one-ended by Lemma \ref{one-ended}. Moreover $G_w$ is clearly a $F_2$-limit group. This follows from the fact that $w$ is not a proper power and therefore we can embed $G_w$ into the extension of centralizer along $\langle w\rangle$ (see for example \cite{bbaumslag}).\\
Let now $\vp: G_w\to F_2$ be a homomorphism. We distinguish two cases:\\
(1) Assume that $w\in \ker \vp$. Then $\vp(A)$ and $\vp(B)$ are cyclic and hence $\vp$ factors through $\eta$.\\
(2) Assume now that $w\notin \ker \vp$. Since by Remark \ref{bem}, $w$ is contained in $[F_2,F_2]$, it follows that $\vp(A)\cong F_2\cong \vp(B)$. Now Corollary \ref{test} yields that there exists $k\in\N$ such that $\vp|_B=\vp|_A\circ c_{w_A^k}$. Hence after precomposing with a Dehn-twist $\alpha$ along $\langle w\rangle$, $\vp\circ\alpha$ factors through $\pi:A\ast_{\langle w\rangle}B\to A$. To be more precise $\alpha:G_w\to G_w$ is given by $\alpha|_A=\id$ and $\alpha|_B=c_{w^k_B}$ and we have then $\vp=\vp|_A\circ\pi\circ\alpha$.
\end{proof}

\noindent Theorem \ref{MR} implies that the Makanin-Razborov diagram of $G_w$ is as in Figure \ref{mr}.

\begin{figure}[htbp]
\centering
\includegraphics[scale=1]{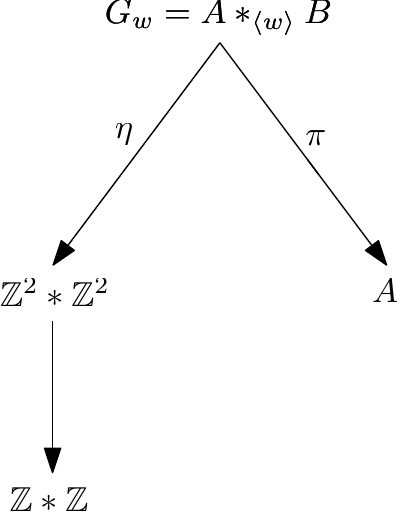}
\caption{The MR diagram for $G_w$}
\label{mr}
\end{figure}

\begin{korollar}
Let $w \in F_2$ be an Ivanov word and $G_w$ be the double of $F_2$ along $w$. Then $G_w$ is not freely subgroup separable.
\end{korollar}

\begin{proof}
Let $g=[b_1,b_2]\in G_w$. Then clearly $g\notin A$ and we claim that $g$ cannot be separated from $A$ in any free quotient. Indeed assume there exists a homomorphism $\vp:G_w\to F$, where $F$ is a non-abelian free group, such that $\vp(g)\notin\vp(A)$. By Theorem \ref{MR}, after possibly precomposing with a Dehn-twist along $\langle w\rangle$, $\vp$ factors through (at least) one of the homomorphisms $\pi$ or $\eta$. Since $\eta(B)\cong\Z^2$ it follows that $\vp(g)=1\in\vp(A)$ if $\vp$ factors through $\eta$. Now assume that $\vp$ factors through $\pi$. But then there exists a Dehn-twist $\alpha$ and a monomorphism $\Psi: A\to F$ such that $\vp(g)=\Psi\circ\pi\circ\alpha(g)\in \Psi\circ\pi(B)=\Psi\circ\pi(A)=\vp(A)$.
\end{proof}

\newpage

\section{JSJ-decompositions of doubles of a free \\ group of rank 2}
Unfortunately we are not able to give a description of all possible Makanin-Razborov diagrams of general doubles of free groups, but at least we are able to classify all possible JSJ decompositions of these doubles. This is the first step towards the construction of more examples of MR diagrams or even to the possible classification of all MR diagrams of doubles of free groups of rank two in the future.\\

In this section we will give a description of the possible JSJ decompositions of (one-ended) doubles of a free group in two generators  (see Theorem \ref{main}).\\
Recall that $G_w$ is the double of $F_2=\langle x_1,x_2\rangle$ along $\Z$ given by the embedding $\iota:\Z\to F_2, 1\mapsto w$, for some $w\in F_2$.
By Lemma \ref{one-ended} $G_w$ is one-ended if and only if $w$ is not contained in a free factor of $F_2$. Therefore if $G_w$ is not one-ended then $G_w$ has a free decomposition of the form
$$G_w\cong\langle y\rangle\ast\langle x\rangle\ast_{\langle w=x^n\rangle}\langle x\rangle\ast\langle y\rangle$$
for some basis $\{x,y\}$ of $F_2$. Hence from now on we will only be interested in one-ended doubles $G_w$ such that in addition $w$ is not a proper power (and therefore $G_w$ is a limit group).\\
First we give some necessary and  sufficient conditions on $w$ such that a JSJ decomposition of $G_w$ is as simple as possible, i.e. it is just the double decomposition $A\ast_CB$.\\
Later we describe JSJ decompositions of $G_w$ in the case that at least one of these conditions is violated.
In the following Proposition we collect some results which will be of use later.

\begin{prop}\label{reference}
Let $G\cong F_2$ be a non-abelian free group on two generators.
\begin{enumerate}[(1)]
\item \cite{solitar} $G$ does not split as a proper amalgamated product along a non-trivial malnormal subgroup $C$.
\item \cite{kapovichweidmann} (Proposition 3.7) Suppose $G$ splits as a non-trivial amalgamated product $A\ast_CB$ along $C\cong\Z$. Then there exist $x,y\in G$ and $n>1$, such that $G=\langle x,y\rangle$, $A=\langle x\rangle\cong\Z$, $B=\langle x^n,y\rangle\cong F_2$ and $A\ast_CB=\langle x\rangle\ast_{\langle x^n\rangle}\langle x^n,y\rangle$. In particular $C$ is malnormal in $B$.
\item \cite{kapovichweidmann} (Lemma 3.6, Proposition 3.8) Suppose $G$ splits as an HNN extension $H\ast_C$ where $C\cong\Z$. Then $G=\langle H, t\ |\ ta^nt^{-1}=b\rangle$ for some $n\geq 1$ and elements $a,b\in G$ with no roots. Moreover $\langle a\rangle$ and $\langle b\rangle$ are malnormal in $H$ and there exist $h\in H$ such that $G=\langle ht,a\rangle$ and $H=\langle a,hbh^{-1}\rangle=\langle a,(ht)a^n(ht)^{-1}\rangle.$
\end{enumerate}
\end{prop}

We say that a group $G$ is obtained by \textit{adjoining a root to $x$} (or \textit{pulling out a root of $x$}) if $G$ has a decomposition as a graph of groups of the form $G=\langle x\rangle\ast_{\langle x^n\rangle}H$.

\begin{prop}\label{jsj}
Let $F_2=\langle x_1,x_2\rangle$ be a free group, $w\in F_2$ an element which is not a proper power and $G_w$ the double of $F_2$ along $w$.
Then the graph of groups given by this decomposition is a JSJ decomposition of $G$ if none of the following holds:
\begin{enumerate}[(1)]
\item $w$ is contained in a subgroup generated by $\{xyx^{-1},y\}$ for some basis $\{x,y\}$ of $F_2$
\item $w$ is contained in a subgroup generated by $\{x^n,y\}$ for some basis $\{x,y\}$ of $F_2$ and some $n\geq 2$.
\end{enumerate}
\end{prop}

\begin{proof}
Assume that $w$ does neither satisfy property $(1)$ nor property $(2)$.
Property (2) implies in particular that $w$ is not a power of a basis element in $F_2$. Hence it follows from Lemma \ref{one-ended} that $G_w$ is one-ended.\\
JSJ-theory implies that either $A\ast_{\langle w\rangle}B$ is a JSJ decomposition of $G_w$ or there exists a $\Z$-splitting of $A$ (respectively $B$) relative $w$.
Therefore it suffices to show that $A$ and $B$ do not admit any further splittings along $\Z$ relative to $\langle w\rangle$.\\
First suppose that $A=R_1\ast_D R_2$ is a non-trivial relative splitting, where $D\cong\Z$.
From Proposition \ref{reference} (2) follows the existence of $r,h\in A$ and $n>1$, such that 
$$F_2\cong A=R_1\ast_D R_2= \langle r\rangle\ast_{\{r^n=b\}} \langle b,h\rangle=\langle r\rangle\ast_{\{r^n=r^n\}}\langle r^n,h\rangle,$$
in particular $A=\langle r,h\rangle$.\\
Since $w$ is not contained in any subgroup generated by $\{x^n,y\}$ for some basis $\{x,y\}$ of $A$ and some $n>1$ this implies that $w\notin R_1, R_2$, a contradiction.\\
Now suppose that $A$ splits as an HNN-extension $R\ast_{\Z}$ relative $w$.
From Proposition \ref{reference} (3) follows that $A=\langle R, t\ |\ ta^nt^{-1}=b\rangle$ for some $n\geq 1$ and elements $a,b\in A$. Moreover there exist $h\in R$ such that $A=\langle ht,a\rangle$ and $R=\langle a,(ht)a^n(ht)^{-1}\rangle.$
Since by assumption $w\notin \langle x,yxy^{-1}\rangle$ for every basis $\{x,y\}$ of $A$, this implies that $w\notin R$.
\end{proof}

We now describe when $G_w$ is a surface group, in which case the JSJ decomposition of $G_w$ consists of a single vertex with QH vertex group $G_w$.

\begin{prop}
Let $G_w$ be the double of $F_2=\langle x_1,x_2\rangle$ along $w$. Then $G_w$ is a surface group if and only if one of the following holds:
\begin{enumerate}[(1)]
\item $w$ is conjugate to $[x_1,x_2]^{\pm1}$ in which case $G_w$ is the fundamental group of an orientable surface of genus 2.
\item $w$ is either conjugate to $x_1^{2\epsilon_1}x_2^{2\epsilon_2}$, where $\epsilon_1,\epsilon_2\in\{\pm1\}$, to $(x_1x_2x_1^{-1}x_2)^{\pm1}$, or to $(x_1x_2x_1x_2^{-1})^{\pm1}$ in which case $G_w$ is the fundamental group of a non-orientable surface of genus 4.
\end{enumerate} 
\end{prop}

\begin{proof}
If $G_w$ is the fundamental group of a closed surface $\Sigma$, then either $\Sigma$ is orientable of genus $2$ or non-orientable of genus $4$. In both cases there exists a basis of $F_2$ such that the curve corresponding to $w$ is conjugate to one of the specified words.
The claim now follows immediately from the fact that the sets $$\{[x_1,x_2]^{\pm1}\}$$ and $$\{x_1^{2\epsilon_1}x_2^{2\epsilon_2}, (x_1x_2x_1^{-1}x_2)^{\pm1},  (x_1x_2x_1x_2^{-1})^{\pm1}\ |\ \epsilon_1,\epsilon_2\in\{\pm 1\}\}$$ are up to conjugation invariant under Nielsen transformations of $F_2$.
\end{proof}

It remains to consider the case that $G_w$ is neither a surface group, nor that the double decomposition of $G_w$ is already a JSJ decomposition.

\begin{satz}\label{main}
Let $G_w$ be the double of $F_2=\langle x_1,x_2\rangle$ along $w$ and assume that $w$ is not a proper power and not contained in a free factor of $F_2$. Suppose moreover that $G_w$ is not a surface group.
\begin{enumerate}[(1)]
\item If $w$ satisfies property $(2)$ from Proposition \ref{jsj}, but not property (1), then a JSJ decomposition $\A$ of $G_w$ has one of the following forms:
\begin{enumerate}[(a)]
\item $\A$ is as in Figure \ref{QH3}.
\item $\A$ is as in Figure \ref{QH3} but the vertices stabilized by $\langle x\rangle$ together with their adjacent edges are replaced by M\"obius strips which are glued along their boundaries to the $4$-punctured sphere. Moreover $m>2$.
\item $\A$ has only rigid vertices and is one of the graphs of groups which we get by refining the vertices $A$ and $B$ in $A\ast_CB$ by one of the following two graphs of groups:
\begin{itemize}
\item $\langle x^n,y\rangle\ast_{\langle x^n\rangle}\langle x\rangle$
\item $\langle y\rangle\ast_{\langle y^m\rangle}\langle x^n,y^m\rangle\ast_{\langle x^n\rangle}\langle x\rangle$.
\end{itemize}
\end{enumerate}
\item If $w$ satisfies property $(1)$ from Proposition \ref{jsj}, but not property $(2)$, then a JSJ decomposition of $G_w$ is the graph of groups which we get by substituting either $\langle xyx^{-1},y\rangle\ast_{\langle y\rangle}$ or one of the three graphs of groups in Figure \ref{HNN-Fall} for $A$ and $B$, or the graph of groups in Figure \ref{QH4}.
\item If $w$ satisfies $(1)$ and $(2)$ from Proposition \ref{jsj} then a JSJ decomposition of $G_w$ is one of the graphs of groups which we get by refining the vertices $A$ and $B$ in $A\ast_CB$ by one of the graphs of groups in Figure \ref{3-Fall}, or the graph of groups in Figure \ref{QH4} with $\gcd(m,n)>1$, or the graph of groups in Figure \ref{QH5}.
\end{enumerate}
\end{satz}

\noindent We split the proof of the theorem in several lemmas. From now on we assume that $w$ is not a proper power and not contained in a free factor of $F_2$. In particular $G_w$ is one-ended by Lemma \ref{one-ended}. Moreover we assume that $G_w$ is not isomorphic to the fundamental group of a closed surface.

\begin{lemma}\label{lemma1}
If $w\in\langle x^n,y\rangle$ for some  $n>1$ and some basis $\{x,y\}$ of $F_2$ and $w\notin\langle bab^{-1},a\rangle$ for any basis $\{a,b\}$ of $F_2$, then a JSJ decomposition $\A$ of $G_w$ has one of the following forms:
\begin{enumerate}[(a)]
\item $\A$ is as in Figure \ref{QH3}.
\item $\A$ is as in Figure \ref{QH3} but the vertices stabilized by $\langle x\rangle$ together with their adjacent edges are replaced by M\"obius strips which are glued along their boundaries to the $4$-punctured sphere. Moreover $m>2$.
\item $\A$ has only rigid vertices and is one of the graphs of groups which we get by refining the vertices $A$ and $B$ in $A\ast_CB$ by one of the following two graphs of groups:
\begin{itemize}
\item $\langle x^n,y\rangle\ast_{\langle x^n\rangle}\langle x\rangle$
\item $\langle y\rangle\ast_{\langle y^m\rangle}\langle x^n,y^m\rangle\ast_{\langle x^n\rangle}\langle x\rangle$.
\end{itemize}
\end{enumerate}
\end{lemma}

\begin{proof}
As in the proof of Proposition \ref{jsj} we conclude that there exists no splitting of $F_2$ as an HNN-extension relative to $w$, but there exists (at least) one splitting of $F_2$ as an amalgamated product relative to $w$. We assume without loss of generality that $w$ is cyclically reduced.\\
Let $n>1$ be maximal such that $\langle x^n,y\rangle\ast_{\langle x^n\rangle}\langle x\rangle$ is such a splitting of $F_2$ relative to $w$ which does not correspond to a splitting along a non-boundary parallel, closed curve on a QH subgroup of $G_w$ (such a splitting exists since by assumption $G_w$ is not a surface group).\\
\textit{Case 1}: Assume that $w\notin \langle x^n,z^m\rangle$ for any $m>1$ and any $z\in F_2$ such that $\{ x,z\}$ is a basis of $F_2$. Suppose that there exists a splitting as an amalgamated product of $\langle x^n,y\rangle$ relative to $w$ and $\langle x^n\rangle$. Hence by using Proposition \ref{reference} (2) we can refine $\langle x^n,y\rangle\ast_{\langle x^n\rangle}\langle x\rangle$ to say
$$\langle e\rangle\ast_{\langle e^k\rangle}H\ast_{\langle x^n\rangle}\langle x\rangle$$
for some $H\cong F_2$ and $k>1$. We denote this graph of groups by $\A$ and the corresponding Bass-Serre tree by $T_A$. We moreover identify $F(x,y)$ with $\pi_1(\A)$.\\
By (the proof of) Theorem $1A$ in \cite{weidmann2} $\{x,y\}$ is Nielsen equivalent to $\{x,\bar{y}\}$ such that either 
\begin{enumerate}[(a)]
\item $T_x\cap T_{\bar{y}}\neq\emptyset$ and $x$ and $\bar{y}$ act elliptically on $T_A$ or
\item $T_x\cap \bar{y}T_x\neq\emptyset$ and $\bar{y}$ acts hyperbolically and $x$ acts elliptically on $T_A$, 
\end{enumerate}
where $T_x:=\{v\in T_A\ |\ x^zv=v\text{ for some }z\in\Z\setminus \{0\}\}$ is the subtree of $T_A$ consisting of all points fixed by a non-trivial power of $x$.\\
Suppose we are in case $(b)$, i.e. $\bar{y}$ acts hyperbolically and $x$ acts elliptically on $T_A$ and $\langle x^n\rangle$ is by Proposition \ref{reference} (2) malnormal in $H$. 
Moreover $\langle x^n\rangle$ is conjugacy separated from $\langle e^k\rangle$, i.e. $g\langle e^k\rangle g^{-1}\cap\langle x^n\rangle=1$ for all $g\in F(x,y)$.
Therefore $T_x$ is the star of the vertex stabilized by $\langle x\rangle$ and all other vertices of $T_x$ are stabilized by conjugates of $H$, while the edges are stabilized by $\langle x^n\rangle$. Clearly $\bar{y}T_x$ is of the same form. Since we are in case (b) we have that $T_x\cap \bar{y}T_x\neq \emptyset$ (see Figure \ref{star}).
\begin{figure}[htbp]
\centering
\includegraphics[scale=1]{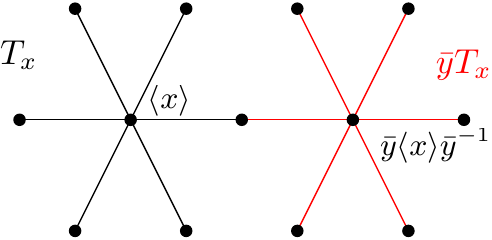}
\caption{$T_x\cup \bar{y}T_x$}
\label{star}
\end{figure}
But then $T:=\langle x,\bar{y}\rangle T_x\subset T_A$ is a $F(x,y)$-invariant subtree, a contradiction to the minimality of $\A$.
Therefore $\bar{y}$ cannot act hyperbolically on $T_A$ and we are in case (a), i.e. $T_x\cap T_{\bar{y}}\neq \emptyset$ and $\bar{y}$ acts elliptically on $T_A$.\\
Hence either $\bar{y}$ is conjugate into $\langle e\rangle$ or into $H$ (since $\bar{y}$ is not conjugate into $\langle x\rangle$). So suppose that $\bar{y}$ is conjugate to an element of $H$. But then mapping $H$, $\langle x\rangle$ to $0$ and $e$ to $1$ defines a non-trivial homomorphism $\vp:F(x,y)\to\Z_k$, such that $F(x,y)=\langle x,\bar{y}\rangle\subset \ker \vp$, a contradiction.\\
Therefore $\bar{y}$ is conjugate to an element of $\langle e\rangle$, hence since $\bar{y}$ is a basis element $\bar{y}=geg^{-1}$ for some $g\in F(x,y)$. Since $T_{\bar{y}}\cap T_x\neq \emptyset$ there exists an isomorphism of graphs of groups from 
$$\langle \bar{y}\rangle \ast_{\langle \bar{y}^k\rangle}\langle x^n,\bar{y}^k\rangle \ast_{\langle x^n\rangle}\langle x\rangle \text{ to } \langle e\rangle \ast_{\langle e^k\rangle}H\ast_{\langle x^n\rangle}\langle x\rangle.$$ In particular there exists an automorphism $\alpha\in \Aut(F_2)$ mapping $\langle \bar{y}^k,x^n\rangle$ onto $H$. But since the images of $e$ and $y=geg^{-1}$ in the abelianization of $F(x,y)$ are equal, this immediately implies that $\alpha$ is an inner automorphism and hence since it fixes $x$ we conclude that $g=x^m$ for some $m\in\N$. Therefore $H=\langle \bar{y}^k,x^n\rangle$ and $\langle \bar{y}\rangle \ast_{\langle \bar{y}^k\rangle}\langle x^n,\bar{y}^k\rangle \ast_{\langle x^n\rangle}\langle x\rangle$ is a splitting of $F(x,y)$ relative $w$, a contradiction to the assumption that $w\notin \langle x^n,\bar{y}^m\rangle$ for any $m>1$.\\
Therefore a JSJ decomposition of $G_w$ is a graph of groups which we get by refining both $A$ and $B$ in $A\ast_CB$ by $\langle x^n,y\rangle\ast_{\langle x^n\rangle}\langle x\rangle$ and all vertex groups are rigid. This is true since $\langle x^n,y\rangle\ast_{\langle w\rangle}\langle x^n,y\rangle$ cannot be a QH subgroup, because then cutting the underlying surface along the curve corresponding to $w$, would yield a surface $S$ with $2$ boundary components and fundamental group $F_2$. Hence by the classification of surfaces $S$ is a punctured M\"obius strip (respectively a twice punctured projective plane, see Figure \ref{QH1}). But this implies that $w\in\langle x^n,y^2\rangle$ for some $n>1$, which has been excluded by assumption. Therefore all vertices in the JSJ decomposition are indeed rigid.\\
\begin{figure}[htbp]
\centering
\includegraphics[scale=1]{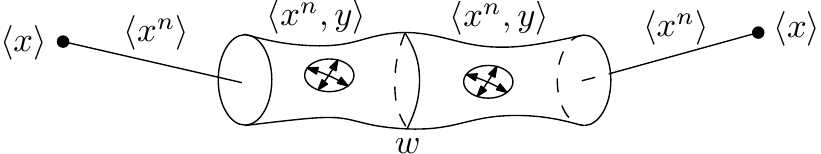}
\caption{$G_w$ does not contain a QH subgroup}
\label{QH1}
\end{figure}
\textit{Case 2}: Now we assume that there exists some (maximal) $m>1$ such that $w\in \langle x^n,y^m\rangle$. Then clearly 
$$\langle y\rangle\ast_{\langle y^m\rangle}\langle x^n,y^m\rangle\ast_{\langle x^n\rangle}\langle x\rangle$$
is a splitting of $F(x,y)$ relative to $w$ and we cannot refine 
$\langle x^n,y^m\rangle$ further as an amalgamated product relative to $\langle x^n\rangle$, $\langle y^m\rangle$ and $w$.\\
Now assume that $w$ is conjugate to $(x^{n\epsilon_1}y^{m\epsilon_2})^{\pm1}$ in $\langle x^n,y^m\rangle$, where $\epsilon_1,\epsilon_2\in\{\pm1\}$. We further distinguish between the following cases depending on $n$ and $m$.\\
First assume that $n,m>2$. Then $\langle x^n,y^m\rangle\ast_{\langle w\rangle}\langle x^n,y^m\rangle$ is the fundamental group of a surface with genus $0$ and $4$ boundary components, two corresponding to $\langle x^n\rangle$ and two corresponding to $\langle y^m\rangle$. And since elements corresponding to essential simple closed curves on surfaces have power at most $2$ in the fundamental group of the surface, the edges in $\langle y\rangle\ast_{\langle y^m\rangle}\langle x^n,y^m\rangle\ast_{\langle x^n\rangle}\langle x\rangle$ do not correspond to simple closed curves on a QH subgroup. Hence the JSJ decomposition of $G_w$ is a graph of groups with one QH vertex group, which is the fundamental group of the $4$-punctured sphere, and four rigid vertices with vertex group isomorphic to $\Z$ connected to the four boundary components (see figure \ref{QH3}).\\
\begin{figure}[htbp]
\centering
\includegraphics[scale=1]{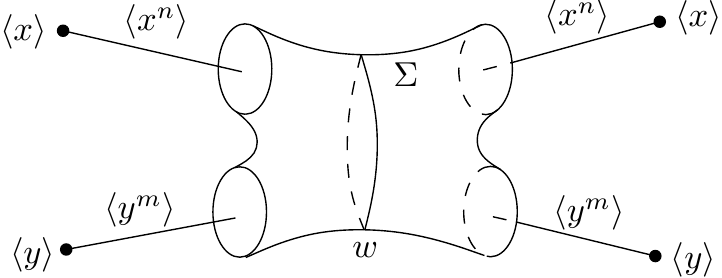}
\caption{The JSJ decomposition of $G_w$ if $w$ is conjugate to $(x^ny^m)^{\pm 1}$}
\label{QH3}
\end{figure}
Now assume that $n>2$ and $m=2$. In this case we can replace the vertices stabilized by $\langle x\rangle$ together with their adjacent edges in Figure \ref{QH3} by M\"obius strips which are glued along their boundaries to the $4$-punctured sphere. Hence a JSJ decomposition of $G_w$ is a graph of groups with one QH vertex group, which is the fundamental group of a twice-punctured Klein bottle $\Sigma$, and two rigid vertices with vertex group isomorphic to $\Z$ connected to the two boundary components of $\Sigma$. The same arguments holds if $n=2$ and $m>2$.\\
So for the last case assume that $n=2=m$. Then we can replace all four vertices with cyclic vertex groups (and their adjacent edges) in figure \ref{QH3} by M\"obius strips which are glued along their boundaries to the boundary components of the $4$-punctured sphere. But in this case clearly $G_w$ is the fundamental group of a non-orientable surface of genus $4$ and this was excluded by the assumption of the theorem.\\
Now assume that $w$ is not conjugate to $(x^{n\epsilon_1}y^{m\epsilon_2})^{\pm1}$, $\epsilon_1,\epsilon_2\in\{\pm1\}$. Then refining $A$ and $B$ by
$$\langle y\rangle\ast_{\langle y^m\rangle}\langle x^n,y^m\rangle\ast_{\langle x^n\rangle}\langle x\rangle$$
yields a JSJ decomposition of $G_w$. In particular all vertex groups are rigid. This holds since contrary to the case above, $\langle x^n,y^m\rangle\ast_{\langle w\rangle}\langle x^n,y^m\rangle$ cannot be a QH subgroup, because otherwise cutting the underlying surface $\Sigma$ along the curve corresponding to $w$, would yield two copies of a surface $S$ with $3$ boundary components and as $\pi_1(S)$ is generated by the elements $x^n$ and $y^m$ corresponding to the boundary components of $\Sigma$, $S$ has genus $0$ (see Figure \ref{QH3}), and the third boundary component is generated by a conjugate of $(x^{n\epsilon_1}y^{m\epsilon_2})^{\pm 1}$ and therefore $w$ is conjugate to $(x^{n\epsilon_1}y^{m\epsilon_2})^{\pm1}$, a contradiction. 
\end{proof}

\begin{lemma}\label{lemma2}
If $w\in\langle xyx^{-1},y\rangle$ for some basis $\{x,y\}$ of $F_2$ and $w\notin\langle a^n,b\rangle$ for any basis $\{a,b\}$ of $F_2$ and any $n>1$, then a JSJ decomposition of $G_w$ has either only rigid vertices and is
one of the graphs of groups which we get by substituting either $\langle xyx^{-1},y\rangle\ast_{\langle y\rangle}$ or one of the three graphs of groups in Figure \ref{HNN-Fall} for $A$ and $B$, 
or has one QH vertex and is the graph of groups in Figure \ref{QH4}.
\end{lemma}

\begin{figure}
\centering
\includegraphics[scale=1]{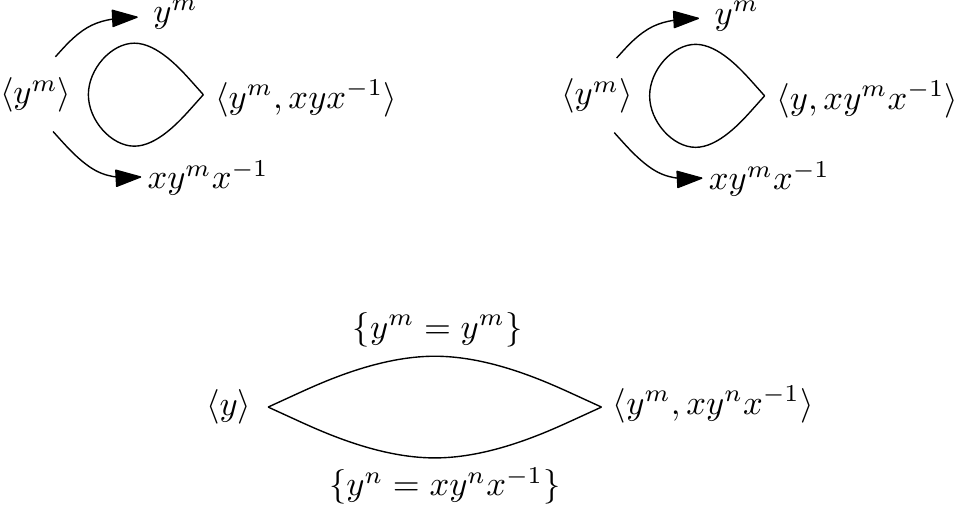}
\caption{The possible relative JSJ decompositions of $F_2$ in the HNN-splitting case}
\label{HNN-Fall}
\end{figure}

\begin{proof}
As in the proof of Proposition \ref{jsj} we conclude that there exists no splitting of $F_2$ as an amalgamated product relative to $w$. But since $w\in\langle xyx^{-1},y\rangle$ for some basis $\{x,y\}$ of $F_2$ there exists (at least) one relative splitting of $F_2$ as an HNN-extension of the form $\langle xyx^{-1},y\rangle\ast_{\langle y\rangle}$ with the embeddings $y\mapsto xyx^{-1}$ and $y\mapsto y$, which does moreover not correspond to a non-boundary parallel closed curve on a surface corresponding to a QH subgroup of $G_w$. The moreover part holds, since by assumption $G_w$ is not a surface group.\\
Since $F_2$ projects onto the fundamental group of the underlying graph of any refinement of $A$, it follows that the first Betti number of any refinement of $A$ is at most one. 
Hence a (relative) refinement of $\langle xyx^{-1},y\rangle$ has to be an amalgamated product, say $X\ast_EY$. By Proposition \ref{reference} we can assume that $X\cong\Z$. Since by assumption there does not exist a splitting of $F(x,y)$ as an amalgamated product relative $w$, we conclude that the refined graph of groups has no separating edge and hence we can assume that either $X=\langle y\rangle$ or $X=\langle xyx^{-1}\rangle$ (see Figure \ref{unfolding}).\\
Moreover we conclude that either $E=\langle y^m\rangle$ or $E=\langle xy^mx^{-1}\rangle$ for some $m>1$. Hence the resulting graph of groups (after plugging in $X\ast_EY$ for $\langle xyx^{-1},y\rangle$) is not reduced and after collapsing the original edge, we get one of the following graphs of groups:
$$\langle xyx^{-1},y^m\rangle\ast_{\langle y^m\rangle}\text{ or }\langle xy^mx^{-1},y\rangle\ast_{\langle xy^mx^{-1}\rangle}$$
and therefore this new graph of groups is just an unfolding of the original one (see Figure \ref{unfolding}).\\
\begin{figure}
\centering
\includegraphics[scale=1]{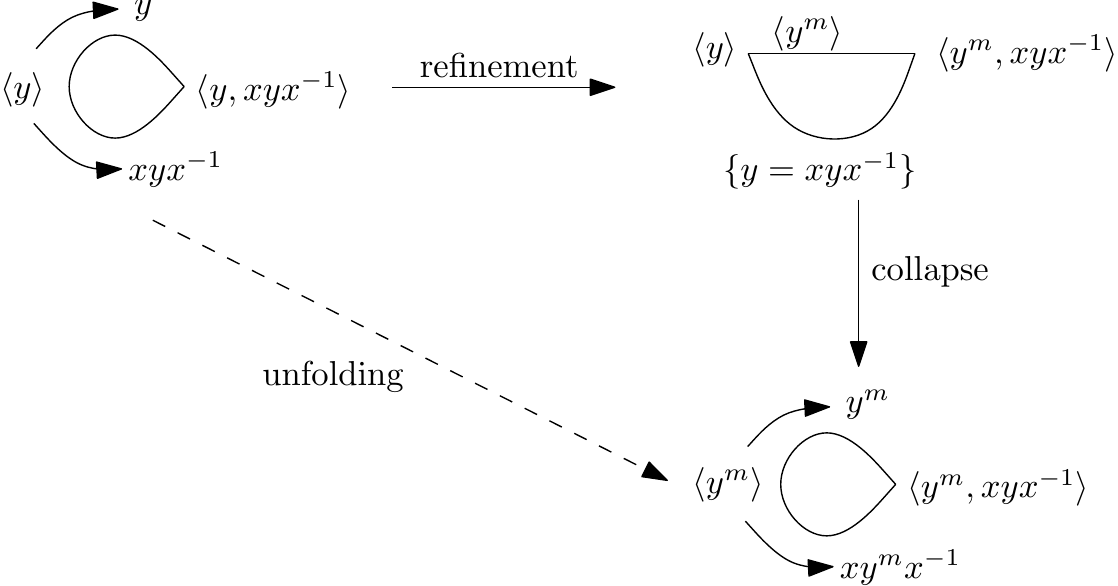}
\caption{Unfolding the graph of groups in the HNN case}
\label{unfolding}
\end{figure}
So if we choose $m,n\geq 1$ maximal such that 
$w\in \langle y^m,xy^nx^{-1}\rangle$ and $w$ is not conjugate to 
$(y^{\pm m}xy^{\pm n}x^{-1})^{\pm 1}$, it follows that (depending on $w$) the only possible relative JSJ decompositions of $F_2$ relative $w$ (up to folding) are the three decompositions shown in Figure \ref{HNN-Fall}. Note that the top two graphs of groups are in fact isomorphic, since both are graphs of groups of the form $$F(a,b)\ast_{\Z}$$ with the embeddings $\alpha: \Z\to \langle a^m\rangle $ and $\omega: \Z\to \langle b^m\rangle$, and $w\in F(a,b)$.
It follows as in the amalgamated product case that all vertex groups are rigid.\\
Note that in the case that $m,n>1$ (and hence the splitting is the bottom graph of groups of Figure \ref{HNN-Fall}),
we have that $k:=\gcd(m,n)=1$ since otherwise we could pull out a root of $y$, i.e. refine the graph of groups by replacing the vertex with vertex group $\langle y\rangle$ by the graph of groups 
$$\langle y\rangle\ast_{\langle y^k\rangle}\langle y^k\rangle$$ and get a non-trivial amalgamated product, a contradiction.\\
So assume now that $m\neq n\geq 1$ are maximal such that $w$ is conjugate to $(y^{\pm m}xy^{\pm n}x^{-1})^{\pm 1}$. Note that again $\gcd(m,n)=1$. Then by the same argument as in the lemma before, $\langle y^m,xy^nx^{-1}\rangle\ast_{\langle w\rangle} \langle y^m,xy^nx^{-1}\rangle$ is the fundamental group of the $4$-punctured sphere, i.e. the surface with genus $0$ and $4$ boundary components, two corresponding to $\langle y^m\rangle$ and two corresponding to $\langle xy^nx^{-1}\rangle$, and hence a JSJ decomposition of $G_w$ is a graph of groups with one QH subgroup, which is the fundamental group of the $4$-punctured sphere, and two rigid vertices with vertex group isomorphic to $\Z$ each connected to two boundary components (see Figure \ref{QH4}).
\begin{figure}[htbp]
\centering
\includegraphics[scale=1]{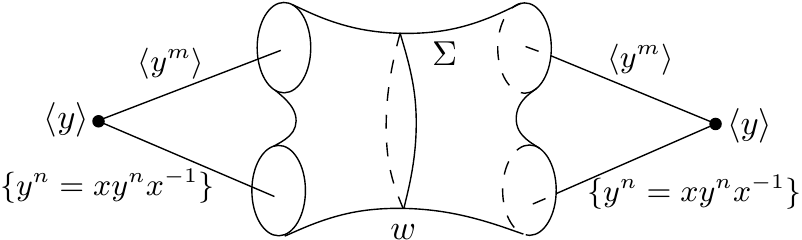}
\caption{The JSJ decomposition of $G_w$ if $w$ is conjugate to $(y^{\pm m}xy^{\pm n}x^{-1})^{\pm 1}$}
\label{QH4}
\end{figure}
\end{proof}

\begin{lemma}\label{lemma3}
If $w\in\langle a^n,b\rangle$ for some  $n>1$ and some basis $\{a,b\}$ of $F_2$ and moreover $w\in\langle xyx^{-1},y\rangle$ for some basis $\{x,y\}$ of $F_2$, then a JSJ decomposition of $G_w$ is one of the graph of groups which we get by refining the vertices $A$ and $B$ in $A\ast_CB$ by one of the graphs of groups in Figure \ref{3-Fall} (in which case all vertices are rigid), the graph of groups in Figure \ref{QH4} with $\gcd(m,n)>1$, or the graph of groups in Figure \ref{QH5}.
\end{lemma}

\begin{proof}
Again it suffices to consider a JSJ decomposition of $F_2$ relative $w$. 
Since $w\in\langle xyx^{-1},y\rangle$ for some basis $\{x,y\}$ of $F_2$, $F_2$ admits in particular a relative splitting which is of one of the three types in Figure \ref{HNN-Fall} with neither the restriction that $\gcd(m,n)=1$ as in Lemma \ref{lemma2} nor that necessarily $m>1$ or $n>1$.\\
\textit{Case 1:} Assume that $F_2$ admits the third type of splitting (the one with two edges). Since the vertex group $\langle y^m, xy^nx^{-1}\rangle$ does not admit any further refinement as an amalgamated product relative $\langle y^m\rangle$, $\langle xy^nx^{-1}\rangle$ and $\langle w\rangle$, it follows that we can pull a root of order $k:=\gcd(m,n)>1$ out of the vertex group $\langle y\rangle$ to refine the graph of groups. The refined graph of groups is the bottom one in Figure \ref{3-Fall} and since there exists no further refinement, this is a JSJ decomposition of $F_2$ relative $w$.\\
\textit{Case 2:} So now assume that $F_2$ admits a splitting relative $w$ of the form:
$$\langle y^n,xyx^{-1}\rangle\ast_{\{ty^nt^{-1}=xy^nx^{-1}\}}$$
or 
$$\langle y,xy^mx^{-1}\rangle\ast_{\{ty^mt^{-1}=xy^mx^{-1}\}}$$
for some maximal $m,n\geq 1$, i.e. we are in one of the two top cases of Figure \ref{HNN-Fall}. By assumption of the lemma there exists a further refinement of the vertex group as an amalgamated product. So let $k>1$ maximal such that $\langle e\rangle\ast_{\langle e^k\rangle} H$ is a splitting of $\langle y,xy^mx^{-1}\rangle$ as an amalgamated product relative to $\langle y^m\rangle$ and $\langle xy^mx^{-1}\rangle$ (we only consider the second case, the first one is similar). We want to show that $e=y$, so let $\A$ be the graph of groups which we get by refining the vertex group in 
$$\langle y,xy^mx^{-1}\rangle \ast_{\{ty^mt^{-1}=xy^mx^{-1}\}}$$
by $\langle e\rangle \ast_{\langle e^k\rangle}H$ and denote by $T_A$ the corresponding Bass-Serre tree. First we note that $x$ acts hyperbolically on $T_A$, since otherwise 
$$x\in\langle e\rangle \ast_{\langle e^k\rangle}H=\langle y,xy^mx^{-1}\rangle$$
and hence $\langle y,xy^mx^{-1}\rangle=F(x,y)$, a contradiction. Moreover $y$ acts elliptically on $T_A$, as $y^m$ does.
As in the proof of Lemma \ref{lemma1} we can deduce from (the proof of) Theorem 1A in \cite{weidmann2} that $\{x,y\}$ is Nielsen equivalent to $\{\bar{x},y\}$ such that either
\begin{enumerate}[(a)]
\item $T_{\bar{x}}\cap T_y\neq\emptyset$ and $\bar{x},y$ act elliptically on $T_A$ or
\item $T_y\cap \bar{x}T_y\neq\emptyset$ and $\bar{x}$ acts hyperbolically, $y$ elliptically on $T_A$.
\end{enumerate}
If $\bar{x}$ acts elliptically on $T_A$ this yields an immediate contradiction to the minimality of $\A$, hence we are in case $(b)$. Since $y$ acts elliptically on $T_A$, it is either conjugate to an element of $H$ or conjugate to $e$.\\
So suppose that $y$ is not conjugate to $e$, i.e. $y$ is conjugate into $H$. Then $T_y\subset T_A$ is the star of the vertex stabilized by $H$ and all other vertices of $T_y$ are stabilized by conjugates of $H$, while the edges are stabilized by $\langle y^m\rangle$. Clearly $\bar{x}T_y$ is of the same form. Since $T_y\cap \bar{x}T_y\neq \emptyset$ we conclude that $\langle \bar{x},y\rangle T_y\subset T_A$ is an $F(x,y)$-invariant subtree, a contradiction to the minimality of $\A$. Hence $y$ is conjugate to $e$ and by the same arguments as in the amalgamated product case (Lemma \ref{lemma1}) it follows without loss of generality that $e=y$ and $k$ divides $n$. Hence $F_2$ has a relative splitting $\A$ which is of one of the two types:\\
$$\langle xyx^{-1}\rangle\ast_{\langle xy^kx^{-1}\rangle}\langle y^n,xy^kx^{-1}\rangle\ast_{\{ty^nt^{-1}=xy^nx^{-1}\}}$$
where $k$ divides $n$, or 
$$\langle y\rangle\ast_{\langle y^k\rangle}\langle y^k,xy^mx^{-1}\rangle\ast_{\{ty^mt^{-1}=xy^mx^{-1}\}}$$
where $k$ divides $m$.
Since we have chosen $k$ maximal, by the same arguments as in the proof of Lemma \ref{lemma1} $\A$ has no further refinement, i.e. $\A$ cannot have one loop edge and two non-loop edges as in Figure \ref{unmoglich}. In addition as in the proof of Lemma \ref{lemma2} these two graphs of groups are isomorphic.\\ 
\begin{figure}[htbp]
\centering
\includegraphics[scale=1]{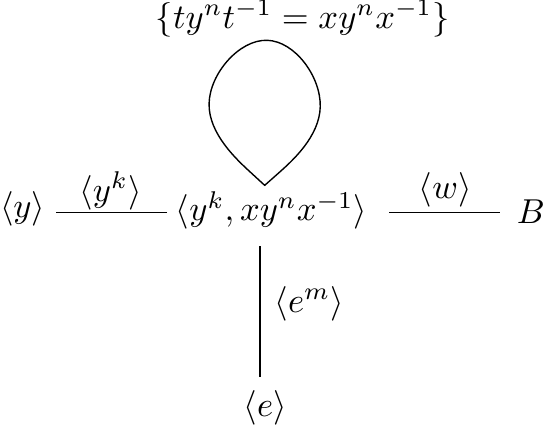}
\caption{$\A$ cannot have one loop edge and two non-loop edges originating from the same vertex}
\label{unmoglich}
\end{figure}
Therefore under the assumptions of Lemma \ref{lemma3} a relative JSJ decomposition of $F_2$ is of one of the two types pictured in Figure \ref{3-Fall}.
\begin{figure}
\centering
\includegraphics[scale=1]{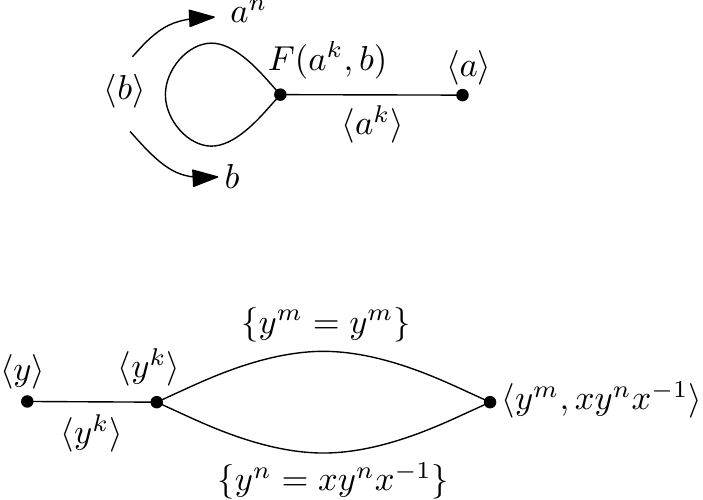}
\caption{The possible relative JSJ decompositions of $F_2$. In the top graph of groups $k$ divides $n$, in the bottom one is $k=\gcd(m,n)$}
\label{3-Fall}
\end{figure}
To complete the proof that a JSJ decomposition of $G_w$ has the desired form, it remains to consider QH vertices. 
We first consider the case that the relative JSJ decomposition of $F_2$ has the form $$\langle y\rangle\ast_{\langle y^k\rangle}\langle y^k,xy^mx^{-1}\rangle\ast_{\{ty^mt^{-1}=xy^mx^{-1}\}}.$$
Note that there exist $l\in\Z$ such that $lk=m$ and therefore 
by sliding the edge corresponding to the loop edge over the non-loop edge we get a graph of groups which has two vertices and two non-loop edges (corresponding to HNN-extensions) connecting these vertices. Hence by the same arguments as in the HNN case (Lemma \ref{lemma2}), we conclude that the JSJ decomposition is as claimed, i.e. with only rigid vertices, if $w$ is not conjugate to $(y^{\pm m}xy^{\pm n}x^{-1})^{\pm 1}$ and as in Figure \ref{QH4} elsewise. In particular in this case there exists a QH vertex group.\\
Now we consider the case that the relative JSJ decomposition is the bottom graph of groups in Figure \ref{3-Fall}. Again by the same arguments as before, we conclude that a JSJ decomposition of $G_w$ is as claimed in the lemma, with only rigid vertices if $w$ is not conjugate to $(y^{\pm m}xy^{\pm n}x^{-1})^{\pm 1}$ and as in Figure \ref{QH5} otherwise. Again in this case there exists a QH vertex group.
\begin{figure}[htbp]
\centering
\includegraphics[scale=1]{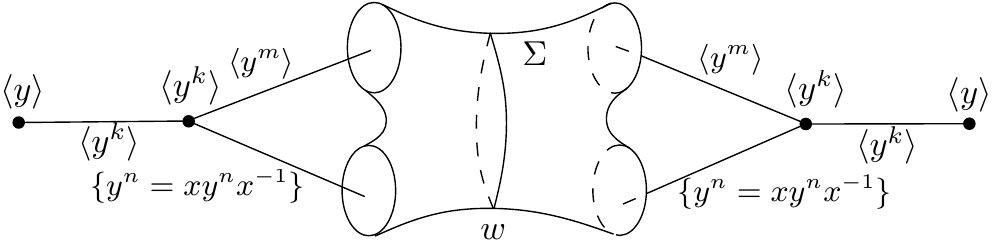}
\caption{The JSJ decomposition of $G_w$ if $w$ is conjugate to $(y^{\pm m}xy^{\pm n}x^{-1})^{\pm 1}$}
\label{QH5}
\end{figure}
\end{proof}

\end{document}